\documentclass[11pt,a4]{amsart}
\usepackage{amsmath}
\usepackage{amsthm}
\usepackage{amssymb}
\usepackage{amscd}
\usepackage{mathrsfs}
\usepackage{txfonts}
\usepackage{graphics,setspace,color}
\newtheorem{theorem}{Theorem}[section]
\newtheorem{proposition}[theorem]{Proposition}
\newtheorem{lemma}[theorem]{Lemma}

\theoremstyle{remark}
\newtheorem{remark}[theorem]{Remark}

\theoremstyle{definition}
\newtheorem{definition}[theorem]{Definition}

\newcommand{\mathsym}[1]{{}}
\newcommand{\unicode}[1]{{}}
\newcommand{\SL}{\mathit{SL}}
\newcommand{\SU}{\mathit{SU}}
\newcommand{\C}{\mathbb{C}}
\newcommand{\R}{\mathbb{R}}
\newcommand{\Q}{\mathbb{Q}}
\newcommand{\Z}{\mathbb{Z}}

\begin{document}

\title
[$\SL(2;\C)$-representations and Reidemeister torsion for homology 3-spheres]
{Some numerical computations on Reidemeister torsion for homology 3-spheres 
obtained by Dehn surgeries along the figure-eight knot}

\author{Teruaki Kitano}

\begin{abstract}
This is the second version of this article. 
There are several errors in the first one. 
They are corrected in this version. 
We show some computations on representations of the fundamental group in $\SL(2;\C)$ 
and Reidemeister torsion for a homology 3-sphere 
obtained by Dehn surgery along the figure-eight knot.
\end{abstract}



\address{Department of Information Systems Science, 
Faculty of Science and Engineering, 
Soka University, 
Tangi-cho 1-236, 
Hachioji, Tokyo 192-8577, Japan}

\email{kitano@soka.ac.jp}

\maketitle

\section{Introduction}

In this note 
we show some numerical computations of representations of the fundamental group 
in $\SL(2;\C)$ 
and Reidemeister torsion for the homology 3-sphere 
obtained by $1/n$-Dehn surgeries along the figure-eight knot. 
More precisely we enumerate all conjugacy classes of irreducible representations 
in $\SL(2;\C)$ 
and compute Reidemeister torsion for these representations.

Reidemeister torsion was originally defined by Reidemeister, Franz and de Rham in the 1930's. 
It can be defined in more general situation, 
but in this paper, 
we consider this invariant for a homology 3-sphere $M$ 
with an irreducible representation $\rho$ of the fundamental group $\pi_1(M)$ into $\SL(2;\C)$. 
It is denoted by $\tau_\rho(M)\in\C$. 

In the 1980's Johnson \cite{Johnson} developed a theory of Reidemeister torsion 
for representations in $\SU(2)$, or in $\SL(2;\C)$. 
That was studied motivated by the relations to the Casson invariant. 
Further he proposed a torsion polynomial of a $3$-manifold. 
In this paper, we define the torsion polynomial as follows. 

Let $M$ be a homology 3-sphere. 
We denote 
the set of conjugacy classes of representations from $\pi_1(M)$ in $\SL(2;\C)$ 
by $\mathcal{R}(M) $ 
and the subset of conjugacy classes with nontrivial value 
of Reidemeister torsion 
by $\mathcal{R}'(M)$. 

Now assume that $\mathcal{R}'(M) $ is a finite set.  

\begin{remark}
In general $\mathcal{R}(M)\ne \mathcal{R}'(M)$. For examples, please see \cite{Johnson, Kitano94-1}. 
\end{remark}

\begin{definition}
A one variable polynomial 
\[
{\sigma}_{M}(t)
=\prod_{[\rho]\in\mathcal{R}'(M)}(t-\tau_\rho(M))
\] is called the torsion polynomial 
of $M_n$. 
\end{definition}

\begin{remark}
If $\mathcal{R}'(M)=\emptyset$, 
then we define 
$\sigma_M(t)=1$.
\end{remark}

In \cite{Kitano15-2, Kitano16-1}
we gave explicit formulas of $\sigma_M(t)$ 
in the case of Brieskorn homology 3-spheres obtained by surgeries along torus knots. 

\begin{remark}
The torsion polynomial was defined as 
\[
\pm
\prod_{[\rho]\in\mathcal{R}'(M)}(t-1/\tau_\rho(M))
\]
under another normalization in \cite{Kitano15-2,Kitano16-1}.
\end{remark}

In this note, 
we show numerical computations for Dehn-surgeries along the figure-eight knot 
by using Mathematica. 

\section{Setting}

First we explain geometric setting, 
which is the same one in \cite{Kitano94-2, Kitano15-1}. 
Please see them for details. 

Let $K\subset S^3$ be the figure-eight knot and $E(K)$ the exterior. 
It is well-known that 
$\pi_1 E(K)$ has the following presentation;
\[
\pi_1 E(K)=\langle
x,y\ |\ wx=yw
\rangle
\]
where 
$w=xy^{-1}x^{-1}y$.

One can take $x$ as a meridian element in $\pi_1 E(K)$. 
As a longitude, one can do 
\[
l=w^{-1}\widetilde{w}
\]
where 
$\widetilde{w}=x^{-1} y x y^{-1}$. 

Let $M_n$ be the homology 3-sphere obtained by $1/n$-Dehn surgery along $K$. 
The fundamental group $\pi_1 M_n$ has the presentation as 
\[
\pi_1 M_n=\langle
x,y\ |\ wx=yw,\ xl^n=1
\rangle.
\]

Let $\rho:\pi_1 M_n\rightarrow \SL(2;\C)$ be an irreducible representation. 
Simply we write $X$ for $\rho(x)$, $Y$ for $\rho(y)$ and so on. 
It is well known that we may assume that $X$ and $Y$ have the following forms;
\[
X=\begin{pmatrix}
 s & 1 \\
 0 & 1/s \\
\end{pmatrix},
Y=\begin{pmatrix}
 s & 0 \\
 -t & 1/s \\
\end{pmatrix}
\]
where 
$s\in\C\setminus\{0\},t\in\C$, after taking conjugations. 

Now define the matrix $R$ by 
$R=W X-Y W$ 
where 
$W=X Y^{-1} X^{-1} Y$. 
The equation $R=\begin{pmatrix}0 & 0\\ 0& 0\end{pmatrix}$ 
induces a system of defining equations of the space of conjugacy classes 
of $\SL(2;\C)$-representations of $\pi_1 E(K)$. 
By direct computations, we have only one equation
\[
f(s,t)=3-\frac{1}{s^2}-s^2+3 t-\frac{t}{s^2}-s^2 t+t^2=0
\]
from $R=\begin{pmatrix} 0 & 0\\0 & 0\end{pmatrix}$. 

Solve $f(s,t)=0$ in $t$, 
\[
t=\frac{1-3 s^2+s^4 \pm \sqrt{1-2 s^2-s^4-2 s^6+s^8}}{2 s^2}.
\]

Hence the parameter $t$ can be eliminated 
by  substituting 
\[
t_\pm=\frac{1-3 s^2+s^4 \pm \sqrt{1-2 s^2-s^4-2 s^6+s^8}}{2 s^2}.
\]

Under $f(s,t)=0$, 
there are four choices on a pair of $(X,Y)$ as follows.  
\begin{enumerate}
\item
$
X=
\begin{pmatrix} s & 1 \\ 0 & 1/s
\end{pmatrix},
Y=
\begin{pmatrix} s & 0 \\ -t_+ & 1/s
\end{pmatrix}$ 
\item
$
X=\begin{pmatrix} s & 1 \\ 0 & 1/s
\end{pmatrix},
Y=\begin{pmatrix} s & 0 \\ -t_- & 1/s
\end{pmatrix}$,
\item
$
X=\begin{pmatrix} 1/s & 1 \\ 0 & s
\end{pmatrix},
Y=\begin{pmatrix} 1/s & 0 \\ -t_+ & s
\end{pmatrix}$ ,
\item
$
X=\begin{pmatrix} 1/s & 1 \\ 0 & s
\end{pmatrix},
Y=\begin{pmatrix} 1/s & 0 \\ -t_- & s
\end{pmatrix}$.
\end{enumerate}

\begin{lemma}
Among the above 4 pairs of $(X,Y)$, 
$(1)$ and $(3)$ give the same conjugacy class and $(2)$ and $(4)$ also gives the same one. 
These two classes are not same. 
\end{lemma}

\begin{proof}
Let us consider the trace of $XY^{-1}$:
$\mathrm{tr}(XY^{-1})=2-t$. 
Then by elementary arguments in linear algebras, 
one can see the above. 
\end{proof}

Solve the inequality 
\[
1-2 s^2-s^4-2 s^6+s^8>0
\]
under the condition $s\neq 0$ 
in the real numbers, 
one has 
\[
\begin{split}
& s\leq -\sqrt{\frac{1}{2} \left(3+\sqrt{5}\right)},
-\sqrt{\frac{1}{2} \left(3-\sqrt{5}\right)}\leq s<0,\\
&0<s\leq \sqrt{\frac{1}{2} \left(3-\sqrt{5}\right)},
\sqrt{\frac{1}{2} \left(3+\sqrt{5}\right)}\leq s
\end{split}
\]
and 
numerically 
\[
s\leq -1.61803, -0.618034\leq s<0,0<s\leq 0.618034, 1.61803\leq s.
\]
If $s$ belongs to the above intervals, 
then the corresponding $(s,t)$ gives an irreducible representation of $\pi_1 E(K)$ in $\SL(2;\R)$. 

Here the matrix corresponding to a longitude is given by 
\[
L=W^{-1}\widetilde{W}=X^{-1} Y X Y^{-1}X^{-1} Y X Y^{-1}.
\]

By direct computations, each entry $L_{ij}(1\leq i,j\leq 2)$ of $L$ is given as follows:
\[
\begin{split}
L_{11}
&=
-1+\frac{1}{2 s^4}-\frac{1}{2 s^2}-\frac{s^2}{2}+\frac{s^4}{2}\pm\frac{1}{2} \sqrt{1-2 s^2-s^4-2 s^6+s^8}\\
& +\frac{\sqrt{1-2 s^2-s^4-2 s^6+s^8}}{2s^4},
\\
L_{12}
&=\pm\frac{\sqrt{1-2 s^2-s^4-2 s^6+s^8}}{s^3}\pm\frac{\sqrt{1-2 s^2-s^4-2 s^6+s^8}}{s},
\\
L_{21}
&=0,
\\
L_{22}
&=-1+\frac{1}{2 s^4}-\frac{1}{2 s^2}-\frac{s^2}{2}+\frac{s^4}{2}\mp\frac{1}{2} \sqrt{1-2 s^2-s^4-2s^6+s^8}\\
& -\frac{\sqrt{1-2 s^2-s^4-2 s^6+s^8}}{2s^4}.
\end{split}
\]
Here the double-sign corresponds in the same order, which is depended on the choice of $t$. 

\begin{remark}
Remark that $\mathrm{tr}(L)$ is not depended on the choice of $t$. 
\end{remark}

We consider a $1/n$-Dehn-surgery along the figure-eight knot. 
Because we do not consider the 3-sphere, then we assume that $n\neq 0$. 
Note that $X$ is corresponding to a meridian. 
Then we compute the relation as 
\[
D=X-L^{-n}=(D_{ij}).
\]

\begin{remark}
It holds that $D_{21}=0$ identically, 
because $X$ and $L$ are upper triangular matrices in $\SL(2;\C)$. 
\end{remark}

Any solution of the systems of equations 
$D_{11}=D_{12}=D_{22}=0$ 
gives a conjugacy class of $\SL(2;\C)$-representation. 
Because $X$ and $L$ are upper triangular matrices, 
then it holds that $D_{11}=0$ for some $s\in\C$ if and only if $D_{22}=0$ for the same $s\in\C$. 
Hence we consider system of two equations $D_{11}=D_{12}=0$. 

\begin{remark}
For any solution $s_0\in\C\setminus\{0\}$, 
then the complex conjugate of $s_0$ is also a solution. 
Because the complex conjugate $\bar{\rho}$ of $\rho$ is alway a representation for any $\rho$. 
If $\rho$ given by $s_0$ is an conjugate to a representation in $\SU(2)$, 
then $\rho$ and $\bar{\rho}$ are conjugate each other. 
Because $\bar{s_0}=1/s_0$. 
\end{remark}

We consider Reidemeister torsion $\tau_\rho(M_n)$ 
for $M_n$ with 
$\rho:\pi_1(M_n)\rightarrow \SL(2;\C)$. 
For the precise definition of Reidemeister torsion $\tau_\rho(M)$ 
for an $\SL(2;\C)$-representation $\rho$, 
please see \cite{Kitano94-1, Kitano94-2, Porti15-1}. 

In the case of $1/n$-surgeries along the figure-eight knot, 
we obtain the following formula 
of Reidemeister torsion in terms of the trace of the meridian image. 

\begin{proposition}[Kitano\cite{Kitano15-2}]
Assume $n\neq 0$. 
If $\rho:\pi_1(M_n)\rightarrow \SL(2;\C)$ is an acyclic representation, 
then one has
\[
\tau_\rho(M_n)
=\frac{2(u-1)}{u^2(u^2-5)}\in\C\setminus\{0\}
\]
where $u=\mathrm{tr}(X)=s+1/s$. 
\end{proposition}

Here $\rho$ is called to be an acyclic representation 
if the chain complex of $M$ with $\C^2_\rho$-coefficients is an acyclic chain complex. 
By numerical computations we compute $\tau_\rho(M_n)$ and $\sigma_{M_n}(t)$ 
by using this formula.  

Here we mention the Casson invariant and the $\SL(;\C)$-Casson invariant. 
Please see \cite{Akbult-McCarthy} and \cite{Curtis, Boden-Curtis06, Boden-Curtis12} 
for precise definitions and properties. 

In 1980's 
Casson defined the Casson invariant $\lambda(M)\in\Z$ 
for a homology 3-sphere $M$ 
as the half of algebraic count of conjugacy classes of irreducible $\SU(2)$-representations. 
In 2000's 
Curtis~\cite{Curtis} defined 
the $\SL(2;\C)$-Casson invariant $\lambda_{\SL(2;\C)}(M)\in\Z$ 
for M  
by counting conjugacy classes of irreducible $\SL(2;\C)$-representations. 
In the case of $1/n$-surgeries along the figure-eight knot, 
one has the following by applying general formula by Casson, and the one by Boden and Curtis. 

\begin{proposition}
\noindent
\begin{itemize}
\item
$\lambda(M_n)=-n.$
\item
$\lambda_{\SL(2;\C)}(M_n)=4n-1.$
\end{itemize}
\end{proposition} 

\begin{remark}
For any positive $n$, 
the above proposition implies that the number of conjugacy classes of $\SU(2)$-representations is algebraically 
$2n$ 
and 
the one of conjugacy classes of $\SL(2;\C)$-representations is $4n-1$. 
\end{remark}

\section{Computation}

Here we show computations from $n=1,\dots,10$ by using Mathematica. 
We make a list of the values of $s$, $u=s+1/s$ and $\tau_{\rho}$. 

We remark the followings. 
\begin{itemize}
\item
We choice the value $s$ in $|s|\leq 1$. 
Because the inverse $1/s$ can be done if $|s|>1$. 
\item
For a representation which is conjugate to the one in $\SU(2)$, 
we choice only one value and omit its complex conjugate. 
\end{itemize}

\subsection{Summary}

We compare computations with the values of Casson invariant $\lambda(M)$ and the $\SL(2;\C)$-Casson invariant 
$\lambda_{\SL(2;\C)}(M)$. 
For any cases, we could find numerically $2|\lambda(M)|$ conjugacy classes in $\SU(2)$ 
and $|\lambda_{\SL(2;\C)}(M)|$ conjugacy classes in $\SL(2;\C)$. 
Further we could also do only one $\SL(2;\R)$-representation. 

We also compute torsion polynomials $\sigma_{M_n}(t)$. 
We simply write $\sigma_n(t)$ to $\sigma_{M_n}(t)$. 
From the definition, 
it is a polynomial over $\Q$, 
because $\tau_\rho(M)$ is an algebraic number for $[\rho]\in\mathcal{R}'(M)$. 
All previous examples in \cite{Kitano15-2,Kitano16-1} are polynomials over $\Z$. 

To compute torsion polynomials, 
values of imaginary parts which are sufficiently small 
as compared with their real parts are regarded as 0.
Because theoretically we can see a torsion polynomial is a polynomial over $\Q$. 

\begin{remark}
In the case of a torus knot, 
there is a 3-term relation among $\sigma_{n+1}(t),\sigma_{n}(t),\sigma_{n-1}(t)$. 
However we see that there is not such a relation in the figure-eight knot case, 
as $\sigma_{-1}(t)=\sigma_1(t),\sigma_0(t)=1$. 
\end{remark}
\subsection{The case of $n= 1$}

The first example $M_1$ is the Briskorn homology 3-sphere $\Sigma(2,3,7)$, 
which is not a hyperbolic manifold. 
Now one has 
\begin{itemize}
\item
$\lambda(M_1)=-1$, 
\item
$\lambda_{\SL(2;\C)}(M_1)=3$ .
\end{itemize}
We find 2 conjugacy classes of $\SU(2)$-representations 
and totally 3 conjugacy classes of $\SL(2;\C)$-representations. 
But the third one that does not come from $\SU(2)$-representations is an $\SL(2;\R)$-representation. 

\begin{table}[htb]
  \begin{tabular}{|c|l|l|l|} \hline
$\SU(2)$ & $s$ & $u=s+1/s$ & $\tau_\rho$  \\ \hline
$\circ$ & $-0.400969+0.916092 i$ & $-0.801938$  & $1.28621$ \\ \hline
$\circ$ &$ 0.277479\, +0.960732 i$ & $0.554958$ & $0.615957$\\ \hline
& $0.611406$ & $2.24698$ & $10.0978$\\ \hline
\end{tabular}
\end{table}

In this case the torsion polynomial is given as 
\[
\sigma_1(t)=t^3-12 t^2+20 t-8. 
\]
This computation coincides with the one in \cite{Kitano15-2}.

\subsection{The case of $n=2$}

The next $M_2$ is a hyperbolic homology 3-sphere. 
One has 
\begin{itemize}
\item
$\lambda(M_2)=-2$, 
\item
$\lambda_{\SL(2;\C)}(M_2)=7$ .
\end{itemize}
In this case we find 4 conjugacy classes of $\SU(2)$-representations 
and totally 7 conjugacy classes of $\SL(2;\C)$-representations. 
The last representation is an $\SL(2;\R)$-representation. 

\begin{table}[htb]
\begin{tabular}{|c|l|l|l|} \hline
$\SU(2)$ & $s$ & $u=s+1/s$ & $\tau_\rho$  \\ \hline
$\circ$ &  $-0.423608+  0.905845 i$ & $-0.847217$ & $1.20196$ \\ \hline
$\circ$  &  $\ -0.194046 +  0.980992 i$ & $-0.388092$ & $3.80096$  \\ \hline
$\circ$ &  $\ \ \ 0.156335\, +  0.987704 i$ & $\ 0.31267$ & $2.86834$  \\ \hline
$\circ$ & $\ 0.476693\, +  0.87907 i$ & $\ 0.953386$ & $0.0250713$  \\ \hline
        &  $-0.69314\pm 0.0194149 i$ & $-2.13472 \mp 0.0209638 i$ & $2.98853 \pm 0.563052 i$  \\ \hline
        &   $0.61642$ & 2.23869 & 42.1266 \\ \hline
\end{tabular}
\end{table}

The torsion polynomial is given by 
\[
\sigma_2(t)
=t^7-56 t^6 +660 t^5 -3384 t^4+ 8720 t^3 -11008 t^2 +5376 t -128  .  
\]

\subsection{The case of $n=3$}

Next one has 
\begin{itemize}
\item
$\lambda(M_3)=-3$, 
\item
$\lambda_{\SL(2;\C)}(M_3)=11$.
\end{itemize}

In this case we find 6 conjugacy classes of $\SU(2)$-representations 
and totally 11 conjugacy classes of $\SL(2;\C)$-representations. 
The last representation is an $\SL(2;\R)$-representation. 

\begin{table}[htb]
\begin{tabular}{|c|l|l|l|}\hline
  $\SU(2)$ & $s$ & $u=s+1/s$ & $\tau_\rho$  \\ \hline
  $\circ$ & $-0.489756+  0.87186 i$ & $-0.979511$ & $1.02124$   \\ \hline
  $\circ$ & $-0.31143+  0.950269 i$ & $-0.622859$ & $1.814$ \\ \hline
  $\circ$ & $-0.125581+  0.992083 i $ & $-0.251162$ & $8.03486$   \\ \hline
  $\circ$ & $0.108419\, + 0.994105 i$ & $0.216838$  & $6.72584$ \\ \hline
  $\circ$ & $0.352641\, +  0.935759 i$ & $0.705282$ & $0.263178$   \\ \hline
  $\circ$ & $0.462835\, +  0.886444 i$ & $0.92567$ & $0.0418747$ \\ \hline
          & $-0.649243\pm 0.009109 i$ & $-2.18919 - 0.0124968 i$ & $6.0064 + 1.53513 i$  \\ \hline
          & $0.762562\, \pm 0.0145425 i$ & $2.07345 - 0.010457 i$ & $-0.709789 + 0.0436681 i$  \\ \hline
          &   $0.61732$ & $2.23723$ & $95.5058$ \\ \hline
\end{tabular}
\end{table}

The torsion polynomial is given by 
\[
\begin{split}
\sigma_3(t)
=&t^{11} -124 t^{10}+3036 t^9 -31696 t^8 +161024 t^7 -364128 t^6\\
& +152640 t^5 
+426752 t^4 -262144 t^3 -142336 t^2 +55296 t-2048.
\end{split}
\]

\subsection{The case of $n=4$}

Next one has 
\begin{itemize}
\item
$\lambda(M_4)=-4$, 
\item
$\lambda_{\SL(2;\C)}(M_4)=15$.
\end{itemize}

In this case we find 8 conjugacy classes of $\SU(2)$-representations 
and totally 15 conjugacy classes of $\SL(2;\C)$-representations. 
The last representation is an $\SL(2;\R)$-representation. 

\begin{table}[htb]
\begin{tabular}{|c|l|l|l|}\hline
   $\SU(2)$ & $s$ & $u=s+1/s$ & $\tau_\rho$  \\ \hline
   $\circ$ & $-0.478316+0.878188 i$ & $-0.956632$  & $1.04682$  \\ \hline
   $\circ$ & $-0.413993+  0.91028 i$ & $-0.827986$ & $1.23604 $  \\ \hline
   $\circ$ &$-0.242737+0.970092 i$ & $-0.485475$ & $2.64583$ \\ \hline
$\circ$ &   $-0.0926466+  0.995699 i$ & $-0.185293$ & $13.9046$   \\ \hline
$\circ$ &$0.0829193\, +0.996556 i$& $0.165839$  & $12.1993$ \\ \hline
$\circ$ &    $0.268854\, +  0.963181 i$ & $0.537707$  & $0.67882$  \\ \hline
$\circ$ & $0.382257\, + 0.924056 i$ & $0.764514$  & $0.182491$ \\ \hline
$\circ$ &    $0.49426\, +  0.869314 i$ & $0.98852$  & $0.00584088$  \\ \hline
 &    $-0.808705\pm 0.0102842 i$ & $-2.04505 \mp 0.00543826 i$ & $1.77945 \pm 0.0421084 i$  \\ \hline
  &    $-0.635184\pm 0.00517794 i$ & $-2.20943 \mp 0.00765508 i$ & $10.2737 \pm 2.88241 i$  \\  \hline
 &    $0.693187\, \pm 0.00964411 i$ & $2.13552 \mp 0.0104227 i$ & $-1.12125 \pm 0.112867 i$  \\ \hline
  &    $0.617633$ & 2.23672 & 170.236 \\ \hline
  \end{tabular}
\end{table}

\[
\begin{split}
\sigma_4(t)
&=
t^{15}-224 t^{14}  + 10320 t^{13}- 211776 t^{12}+ 2.2964 \times 10^6 t^{11} - 1.35709 \times 10^7  t^{10}\\
&+ 4.11722 \times 10^7t^9- 4.96721 \times 10^7 t^8- 3.55295 \times 10^7 t^7  + 1.56351 \times 10^8 t^6 \\
&- 1.13653 \times 10^8 t^5 - 5.89578 \times 10^7 t^4  + 1.15933 \times 10^8 t^3   - 5.0004 \times 10^7 t^2\\
&+5.89824 \times 10^6 t -32768
\end{split}
\]

\subsection{The case of $n=5$}

Next one has 
\begin{itemize}
\item
$\lambda(M_5)=-5$, 
\item
$\lambda_{\SL(2;\C)}(M_5)=19$.
\end{itemize}

In this case we find 10 conjugacy classes of $\SU(2)$-representations 
and totally 19 conjugacy classes of $\SL(2;\C)$-representations. 
The last representation is an $\SL(2;\R)$-representation. 

\begin{table}[htb]
  \begin{tabular}{|c|l|l|l|} \hline
$\SU(2)$ &   $s$ & $u=s+1/s$ & $\tau_\rho$  \\ \hline
$\circ$   & $-0.496333+  0.868132 i$ & $-0.992666$ & $1.00743$  \\ \hline
$\circ$  & $ -0.420075+0.907489 i $ & $-0.840151$ & $1.21421$ \\ \hline
$\circ$   & $-0.343785+  0.939049 i$ & -0.687569 & $1.57697$  \\ \hline
$\circ$ & $-0.19818+0.980166 i$ & $-0.396359$ & $3.67066$ \\ \hline
$\circ$   & $-0.073359+  0.997306 i$ & $-0.146718$ & $21.4005$ \\ \hline
$\circ$  & $ 0.0671124\, +0.997745 i $ & $0.13422$ & $19.2915$ \\ \hline
$\circ$   & $0.215697\, +  0.97646 i$ & $0.431395$ & $1.26939$   \\ \hline
$\circ$  & $ 0.319062\, +0.947734 i $ & $0.638124$ & $0.386993$ \\ \hline 
$\circ$   & $0.44386\, +  0.896096 i$ & $0.88772$ & $0.0676542$   \\ \hline
$\circ$  & $ 0.485869\, +0.874031 i$ & $0.971739$ & $0.0147589$  \\ \hline
      & $-0.628893\pm 0.00332617 i$ & $-2.21894 \mp 0.00508349 i$ & $15.7706 \pm 4.61071 i$  \\ \hline
   & $-0.731102\pm 0.00855695 i$ & $-2.09871 \mp 0.00744981 i$ & $2.35697 \pm 0.11268 i$  \\ \hline
   & $0.840595\, \pm 0.00745097 i$ & $2.03014 \mp 0.00309303 i$ & $-0.568873 \pm 0.00810636 i$\\ \hline   
   & $0.664373\, \pm 0.00643176 i$ & $2.16941 \mp 0.00813843i$ & $-1.66792 \pm 0.199595 i$  \\ \hline
   & $0.617778$ & $2.23648$ & $266.318$  \\ \hline 
\end{tabular}
\end{table}

The torsion polynomial is given by 
\[
\begin{split}
\sigma_5(t)
&=
t^{19}- 348 t^{18}+ 24428 t^{17}- 756768 t^{16}+ 1.22252 \times 10^7 t^{15} - 1.05049 \times 10^8 t^{14}\\
&
+ 4.39482 \times 10^8 t^{13} - 6.05556 \times 10^8 t^{12}- 1.45911 \times 10^9 t^{11} + 5.73225 \times 10^9 t^{10}\\
&- 3.56966 \times 10^9 t^9 - 9.32096 \times 10^9 t^8
+ 1.48101 \times 10^{10} t^7 - 1.9304 \times 10^9 t^6 \\
&  - 7.91541 \times 10^9 t^5 +3.44198 \times 10^9 t^4+ 1.05592 \times 10^9 t^3  - 6.28883 \times 10^8 t^2\\
&+ 4.45645 \times 10^7 t 
  -524288 
\end{split}
\]

\subsection{The case of $n=6$}

One has 
\begin{itemize}
\item
$\lambda(M_6)=-6$, 
\item
$\lambda_{\SL(2;\C)}(M_6)=23$.
\end{itemize}

In this case we find 12 conjugacy classes of $\SU(2)$-representations 
and totally 23 conjugacy classes of $\SL(2;\C)$-representations. 
The last representation is an $\SL(2;\R)$-representation. 

\begin{table}[htb]
  \begin{tabular}{|c|l|l|l|}\hline
$\SU(2)$ &  $s$ & $u=s+1/s$ & $\tau_\rho$  \\ \hline
$\circ$  & $ -0.490087+0.871673 i$ & $-0.980174$ & $1.02053$\\ \hline
$\circ$  &  $-0.460559+  0.887629 i$ & $-0.921118$ & $1.0908$  \\ \hline
$\circ$ & $ -0.366341+0.93048 i $ & $-0.732683$ & 1.44635 \\ \hline
$\circ$   &  $-0.290911+  0.95675 i$ & -0.581822 & $2.00486$  \\ \hline
  $\circ$  &$ -0.167221+0.985919 i$ & $-0.334442$ & $4.8814$ \\ \hline
  $\circ$   &  $-0.0607063+  0.998156 i$ & $-0.121413$ & $30.5197$   \\    \hline
  $\circ$  & $ 0.0563605\, +0.99841 i$ & $0.112721$ & $28.0037$ \\ \hline
$\circ$  &  $0.179677\, +  0.983726 i$ & $0.359355$  & $2.03702$ \\ \hline
$\circ$  & $0.272216\, +0.962236 i$ & $0.544432$ & $0.653532$ \\ \hline
$\circ$   &  $0.387688\, +  0.921791 i$ & $0.775376$  & $0.169874$  \\ \hline
$\circ$  & $ 0.442457\, +0.89679 i$ & $0.884914$ & $0.0697032$ \\ \hline
$\circ$   &  $0.497456\, +  0.867489 i$ & $0.994912$ & $0.00256369$   \\   \hline
&  $-0.625531\pm 0.00231384 i$ & $-2.22415 \mp 0.00359944 i$ & $22.4926 \pm 6.72158 i$  \\ \hline
&  $-0.693197\, \pm 0.00642162 i$ & $-2.13566 \mp 0.00694108 i$ & $3.11879 \pm 0.197207 i$  \\ \hline  
&  $-0.863685\pm 0.00558142 i$ & $-2.02147 \mp 0.00190054 i$ & $1.61843 \pm 0.0115852 i$  \\ \hline
&  $0.762163\, \pm 0.00726544 i$ & $2.0741 \mp 0.00524079 i$ & $-0.714588 \pm 0.0221273 i$  \\ \hline
&  $0.64957\, \pm 0.00453793 i$ & $2.18897 \mp 0.00621643 i$ & $-2.34127 \pm 0.304635 i$  \\ \hline
&  $0.617856$ & 2.23636 & 383.752  \\ \hline
  \end{tabular}
\end{table}

The torsion polynomial is given by 
\[
\begin{split}
\sigma_6(t)
&=
t^{23}- 504 t^{22}+ 52020 t^{21} - 2.40364 \times 10^6 t^{20} + 5.93684 \times 10^7 t^{19}\\
&- 8.20377 \times 10^8 t^{18}+6.23643 \times 10^9 t^{17}- 2.42844 \times 10^{10} t^{16} + 2.55758 \times 10^{10} t^{15}\\
& + 1.64156 \times 10^{11} t^{14} - 6.99939 \times 10^{11} t^{13}+ 8.04666 \times 10^{11} t^{12}+ 1.36418 \times 10^{12} t^{11}\\
&- 5.35777 \times 10^{12} t^{10}+ 6.06942 \times 10^{12} t^9 - 6.38688 \times 10^{11} t^8 - 6.38688 \times 10^{11} t^8\\
&- 4.81682 \times 10^{12} t^7  + 4.04757 \times 10^{12} t^6 - 2.98072 \times 10^{11} t^5 - 1.0991 \times 10^{12} t^4\\
&+ 5.29833 \times 10^{11}t^3 - 7.9675 \times 10^{10} t^2 + 3.47288 \times 10^9 t
-8.38861 \times 10^6  .
 \end{split}
\]

\subsection{The case of $n=7$}

First one has $\lambda(M_7)=-7$, $\lambda_{\SL(2;\C)}(M_7)=27$. 

In this case we find 14 conjugacy classes of $\SU(2)$-representations like previous examples. 
We do more 13 conjugacy classes of $\SL(2;\C)$-representations. 
Among them we do only one $\SL(2;\R)$-representation. 

\begin{table}[htb]
  \begin{tabular}{|c|l|l|l|} \hline
 $\SU(2)$ &  $s$ & $u=s+1/s$ & $\tau_\rho$  \\ \hline
$\circ$ &    $-0.498132 +  0.867101 i$ & $-0.996264$ & $1.00376$   \\     \hline
$\circ$ &    $-0.456717+  0.889612 i$ & $-0.913434$ & $1.10105$  \\ \hline
$\circ$ &    $-0.415538 +  0.909576 i$ & $-0.831076$ & 1.2304   \\   \hline
$\circ$ &    $-0.322425+0.946595 i$ & $-0.644849$ & $1.725754883584166$ \\ \hline 
$\circ$ &    $-0.251181 +  0.96794 i$ & $-0.502362$ & $2.50781$  \\       \hline
$\circ$ &    $-0.144537+  0.989499 i$ & $-0.289074$ & $6.27537$  \\ \hline
$\circ$ &    $-0.0517713 +  0.998659 i$ & $-0.103543$ & $41.2613$  \\  \hline
$\circ$ &    $0.0485751\, + 0.99882 i$ & $0.0971502$ & $38.3362$  \\ \hline
$\circ$ &    $0.153815\,  +  0.9881 i$ & $0.30763$  & $2.98291$   \\    \hline
$\circ$ &    $0.236804\, +  0.971558 i$ & $0.473608$  & $0.982802$  \\ \hline
$\circ$ &    $0.340083\,  +  0.940396 i$ & $0.680166$  & $0.304734$  \\   \hline 
$\circ$ &    $0.397537\, +  0.917586 i$ & $0.795074$  & $0.148438$  \\  \hline
$\circ$ &    $0.470809\,  +  0.882235 i$ & $0.941618$ & $0.0320157$   \\ \hline
$\circ$ &    $0.492671\, +  0.870215 i$ & $0.985343$ & $0.00749368$ \\ \hline
&    $-0.623523\pm 0.001701i$ & $-2.22730105 \mp 0.002675 i$ & $30.4388 \pm 9.21537 i$  \\ \hline
&    $-0.671806\pm 0.00487214 i$ & $-2.16025 \mp 0.00592043 i$ & $4.04034 \pm 0.29538 i$ \\ \hline
&    $-0.787517 \pm 0.00610802 i$ & $-2.05725 \mp 0.00374012 i$ & $1.88117 \pm  0.0331674 i$  \\ \hline
&    $0.881081\,  \pm 0.00431162 i$ & $2.01602 \mp 0.00124229 i$ & $-0.534338 \pm 0.00285536 i$  \\ \hline 
&    $0.719217\,  \pm 0.00597955 i$ & $2.10952 \mp 0.00557931 i$ & $-0.905083 \pm 0.0385034 i$  \\ \hline
&    $1.560218\pm 0.008182 i$ & $2.201129\pm 0.0048120i$ & $-3.13933 \mp 0.427727 I$ \\ \hline
 &    $0.617903$ & $2.23628$ & $521.407$ \\ \hline
\end{tabular}
\end{table}

Now we can see that 
\[
\begin{split}
\sigma_7(t)
&=t^{27} -684 t^{26}+94916 t^{25}-5.8661\times 10^6 t^{24}+1.92341\times 10^8 t^{23}-3.348194\times 10^9 t^{22}\\
&+3.38294\times 10^{10} t^{21} -1.60441\times 10^{11} t^{20} +1.21022\times 10^{11} t^{19}+ 2.40426\times 10^{12} t^{18}\\
&-1.09427 \times 10^{13} t^{17} 
+ 1.08363 \times 10^{13} t^{16}  +5.40471\times 10^{13}  t^{15}-1.97613\times 10^{14} \times 10^7 t^{14} \\
&+2.18273 \times 10^{14} t^{13}
+1.14878 \times 10^{14} t^{12}  -4.92745 \times 10^{14} t^{11}+3.10448 \times 10^{14} t^{10}\\
&+2.31893 \times 10^{14} t^9
-3.52574\times 10^{14} t^8+4.61951\times 10^{13} t^7 +1.12401 \times 10^{14} t^6\\
& -4.44561 \times 10^{13} t^5-8.4299 \times 10^{12} t^4+6.16006\times 10^{12} t^3\\
& -7.89939\times 10^{11} t^2 +2.34881\times 10^{10} t ^1 -1.34218\times 10^8.
\end{split}
\]
\subsection{The case of $n=8$}

In the next case, one has 
\begin{itemize}
\item
$\lambda(M_8)=-8$, 
\item
$\lambda_{\SL(2;\C)}(M_8)=31$.
\end{itemize}

In this case we find 16 conjugacy classes of $\SU(2)$-representations 
and totally 31 conjugacy classes of $\SL(2;\C)$-representations. 
The last representation is an $\SL(2;\R)$-representation. 

\begin{table}[htb]
  \begin{tabular}{|c|l|l|l|}\hline
 $\SU(2)$ &  $s$ & $u=s+1/s$ & $\tau_\rho$  \\ \hline
 $\circ$ & $-0.494366+0.869254 i$ & $-0.988732$ & $1.011492842$\\ \hline
 $\circ$ &$ -0.47754+  0.87861 i$ & $-0.95508$ & $1.048631249$ \\ \hline
 $\circ$ & $-0.419132+0.907925 i$ & $-0.838263$ &  $1.217529795$\\ \hline
 $\circ$ &  $ -0.37393+  0.927457 i$ & $-0.747861$ & $1.407484041$\\ \hline
 $\circ$ & $ -0.28699+0.957934 i$ & $-0.573979$ & $2.045827739$\\ \hline
$\circ$ &  $-0.220616+  0.975361 i$ & $-0.441231$ & $3.081126266$ \\ \hline
$\circ$ & $ -0.127229+0.991873 i$ &$-0.254459$ & $7.85131636$\\ \hline
$\circ$ &  $-0.0451269+  0.998981 i $ & $-0.0902538$ & $53.6247182$\\ \hline
$\circ$ & $ 0.042678\, +0.999089 i$ & $0.0853561$ & $50.2893844$\\ \hline
$\circ$ &$0.134395\, +  0.990928 i$ & $0.268789$ & $4.107704808$ \\ \hline
$\circ$ & $0.209299\, +0.977852 i$ & $0.418599$ & $1.375414532$\\ \hline
$\circ$ &$0.301435\, +  0.953487 i$ & $0.60287$ & $0.471324314$ \\ \hline
$\circ$ & $0.357963\, +0.933736 i$ & $0.715926$ & $0.247016780$\\ \hline
$\circ$ &$0.434268\, +  0.900784 i $ & $0.868535$ & $0.0820955924$ \\ \hline
$\circ$ & $0.466323\, +0.884614 i$ & $0.932646$ & $0.0374963906$\\ \hline
$\circ$ &$0.49857\, +  0.866849 i $ & $0.997141$ & $0.0014358346$ \\ \hline
& $ -0.622227\pm 0.00130353 i$ & $-2.22935\mp 0.0020633 i$ & $39.60771 \pm 12.09283 i$\\ \hline
& $ -0.658499\pm 0.0037911 i$ & $-2.17705\mp 0.00495151 i$ & $5.1134622 \pm 0.4078733 i$\\ \hline  
& $-0.742103\pm 0.00541306 i$ & $-2.08955\mp 0.00441555 i$ & $2.2312104 \pm 0.05871724 i$\\ \hline
& $-0.894615\pm 0.00341956 i$ & $-2.0124\mp 0.000853032 i$& $1.56555866 \pm 0.00477245 i$\\ \hline
& $0.808382\, \pm 0.0051452 i $ &$2.04537\, \mp 0.002728 i$ & $-0.61198103 \pm 0.008329177 i$\\ \hline 
& $0.693201\, \pm 0.00481418 i$ & $2.13571\, \mp 0.00520388 i$& $-1.1321183 \pm 0.05702512  i$ \\ \hline  
& $0.635422\, \pm 0.00258348 i$ & $2.20915\, \mp 0.00381495 i$ & $-4.060877 \pm 0.5708212  i$\\ \hline
& $0.617934$ & $2.23623$ & $682.674$\\ \hline
\end{tabular}
\end{table}

The torsion polynomial is given by the following. 

\[
\begin{split}
\sigma_8(t)
&=
t^{31}-896 t^{30}+164160 t^{29}-1.34889\times
10^7 t^{28}+5.9497\times 10^8 t^{27}-1.48208\times
10^{10} t^{26}\\
&+2.0792\times 10^{11} t^{25}-1.5975\times
10^{12} t^{24}+5.113\times 10^{12} t^{23}+1.609\times
10^{13} t^{22}\\
&-2.199\times 10^{14} t^{21}+7.892\times
10^{14} t^{20}-2.13\times 10^{14} t^{19}-8.29\times
10^{15} t^{18}\\
&+3.230\times 10^{16} t^{17} -5.36\times
10^{16} t^{16}+1.31\times 10^{16} t^{15}+1.119\times
10^{17} t^{14}\\
&-1.979\times 10^{17} t^{13}+8.35\times
10^{16} t^{12}+1.488\times 10^{17} t^{11}-2.212\times
10^{17} t^{10}\\
&+7.22\times 10^{16} t^9+7.11\times
10^{16} t^8-7.047\times 10^{16} t^7+1.329\times
10^{16} t^6+9.403\times 10^{15} t^5\\
&-5.4486\times
10^{15} t^4+1.03031\times 10^{15} t^3-7.2572\times
10^{13} t^2\\
&+1.59773\times 10^{12} t-2.14748\times
10^9
\end{split}
\]
%
\subsection{The case of $n=9$}

First one has 
\begin{itemize}
\item
$\lambda(M_9)=-9$, 
\item
$\lambda_{\SL(2;\C)}(M_9)=35$. 
\end{itemize}

Here we find 18 conjugacy classes of $\SU(2)$-representations 
and totally 35 conjugacy classes of $\SL(2;\C)$-representations.  
The last one is also a $\SL(2;\R)$-representation. 

\begin{table}[htb]
  \begin{tabular}{|c|l|l|l|} \hline
 $\SU(2)$ &  $s$ & $u=s+1/s$ & $\tau_\rho$  \\ \hline
$\circ$ &$-0.498871+  0.866676 i $ & $-0.9977413535$ & $1.002267596$\\ \hline
$\circ$ &$ -0.473084+0.881017 i$ & $-0.946168$ & $1.059217610 $ \\ \hline
$\circ$ &$-0.447444+  0.894312 i$ & $-0.894888$ & $1.126969070$ \\ \hline
$\circ$ &$ -0.38399+0.923337 i$ & $-0.76798$ & $1.359403531$\\ \hline
$\circ$ &$-0.338051+  0.941128 i$ & $-0.676101$ & $1.614264278$\\ \hline
$\circ$ &$ -0.258146+0.966106 i$ & $-0.516291$ & $2.403508579$ \\ \hline
$\circ$ &$-0.196502+  0.980503 i$ & $-0.393003$ & $3.722604536 $\\ \hline
$\circ$ &$ -0.113601+0.993526 i$ & $-0.227202$ & $9.60858961$\\ \hline
$\circ$ &$-0.0399928+  0.9992 i$ & $-0.0799857$ & $67.6097739$\\ \hline
$\circ$ &$ 0.038057\, +0.999276 i $ & $0.0761141$ & $63.8633816$\\ \hline
$\circ$ &$ 0.119296\, +  0.992859 i$ & $0.238591$ & $5.41179913$ \\ \hline
$\circ$ & $ 0.187397\, +0.982284 i$ & $0.374793 $ & $1.831789427$ \\ \hline
$\circ$ &$ 0.270042\, +  0.962849 i$ & $0.540083 $ & $0.669766151$\\ \hline
$\circ$ &$ 0.324317\, +0.945948 i$ & $0.648634$ & $0.364749065$ \\ \hline
$\circ$ &$0.398186\, +  0.917305 i$ & $0.796372$ & $0.1470861889$\\ \hline
$\circ$ &$ 0.434658\, +0.900595 i$ & $0.869317 $ & $0.0814873088$ \\ \hline
$\circ$ &$0.482193\, +  0.876065 i$ & $0.964385$ & $0.0188178181$\\ \hline
$\circ$ &$ 0.495536\, +0.868587 i$ & $0.991072$ & $0.0045248460$ \\ \hline
& $ -0.621342 \pm 0.00103034 i$ & $-2.230757 \mp 0.001638478 i$ & $49.99999 \pm 15.35354 i$  \\ \hline 
& $-0.649630 \pm 0.00302325 i$ & $-2.18893 \mp 0.00414036 i$ & $6.335051 \pm 0.5346564 i$\\ \hline
& $-0.712996 + 0.00458062 i$ & $-2.115470751 \mp 0.004429539 i$ & $2.6499001 \pm 0.08729179 i$\\ \hline
& $ -0.825748\pm 0.00436242 i$ & $-2.03674 \mp 0.00203523 i$ & $1.71886598 \pm 0.01444799 i$  \\ \hline 
& $ 0.905426\, \pm 0.00277285 i$ & $2.00987 \mp 0.000609481 i$ & $-0.52058388 \pm 0.001326408 i$\\ \hline
& $0.76209 \pm 0.00484289 i$ & $2.074218042 \mp 0.003495356 i$ & $-0.71548730 \pm 0.014788047 i$ \\ \hline
& $0.676202 \pm 0.00390501 i$ & $2.155000405 \mp 0.004634916 i$ & $-1.3928905 \pm 0.07775814 i$\\ \hline
& $0.631699 \pm 0.00204673 i$ & $2.214715401 \mp 0.003082315 i$ & $-5.106041 \pm 0.7321823 i$\\ \hline
 &$ 0.6179549394855183$ & $2.236195908$ & $864.16$\\ \hline
\end{tabular}
\end{table}

The torsion polynomial is given by 
\[
\begin{split}
\sigma_9(t)
&=
t^{35}-1132 t^{34}+260580t ^{33}-2.68147\times 10^7 t^{32}+1.47405\times 10^9 t^{31}\\
& -4.5380\times 10^{10} t^{30}+7.7496\times 10^{11} t^{29}-7.064\times 10^{12} t^{28}+2.387\times 10^{13} t^{27}\\
&+1.509\times 10^{14} t^{26}-1.968\times 10^{15} t^{25}+7.51\times 10^{15} t^{24}+4.22\times 10^{15} t^{23}\\
&-1.523\times 10^{17} t^{22}+6.24\times 10^{17} t^{21}-1.058\times 10^{18} t^{20}-2.9\times 10^{17} t^{19}\\
&+4.70\times 10^{18}  t^{18}- 7.64\times 10^{18} t^{17}+ 5.\times 10^{17} t^{16}+1.39\times 10^{19} t^{15}\\
&-1.62\times 10^{19} t^{14}- 1.3\times 10^{18} t^{13} +1.61\times 10^{19} t^{12}- 9.5\times 10^{18} t^{11}\\
&-4.19\times 10^{18} t^{10}+ 6.28\times 10^{18} t^9-9.4\times 10^{17} t^8-1.364\times 10^{18} t^7\\
&+ 5.645\times 10^{17} t^6+4.88\times 10^{16} t^5- 6.157\times 10^{16} t^4+1.06336\times 10^{16} t^3\\
&-6.2396\times 10^{14} t^2+ 1.02048\times 10^{13} t-3.4360\times
10^{10}
\end{split}
\]

\subsection{The case of $n=10$}

One has 
\begin{itemize}
\item
$\lambda(M_{10})=-10$, 
\item
$\lambda_{\SL(2;\C)}(M_9)=39$.
\end{itemize}

In this case we find 20 conjugacy classes of $\SU(2)$-representations. 
We do totally 39 conjugacy classes of $\SL(2;\C)$-representations. 
The last one is an $\SL(2;\R)$-representation. 

\begin{table}[htb]
  \begin{tabular}{|c|l|l|l|}\hline
 $\SU(2)$ &  $s$ & $u=s+1/s$ & $\tau_\rho$  \\ \hline
 $\circ$ &$ -0.496377+0.868107 i$ & $-0.992753$ & $1.007339322 $\\ \hline
 $\circ$ &$ -0.48554+  0.874215 i$ & $-0.971079$ & $1.030431375$\\ \hline
 $\circ$ &$ -0.446171+0.894948 i$ & $-0.892342$ & $1.130661131$\\ \hline
 $\circ$ & $-0.416138+  0.909302 i$ & $-0.832276$ & $1.228229667$\\ \hline
 $\circ$ &$ -0.352798+0.9357 i$ & $-0.705596$ & $1.521863518$\\ \hline
$\circ$ & $ -0.307592+  0.951518 i $ &$-0.615185$ & $1.846942653$\\ \hline
$\circ$ &$ -0.234348+0.972153 i$ & $-0.468696$ & $2.797187148$\\ \hline
$\circ$ &$ -0.177045+  0.984203 i$ & $-0.35409$ & $4.43107072$\\ \hline
$\circ$ &$ -0.102597+0.994723 i$ & $-0.205194$ & $11.54680784$\\ \hline
$\circ$ &$ -0.035907+  0.999355 i$ & $-0.0718141$ & $83.2162915$\\ \hline
$\circ$ &$ 0.0343385\, +0.99941 i$ & $0.068677$ & $79.0582830$\\ \hline
$\circ$ &$ 0.107229\, +  0.994234 i$ & $0.214457$ & $6.89542715$ \\ \hline
$\circ$ &$ 0.169577\, +0.985517 i$ & $0.339153$ & $2.352207686$\\ \hline
$\circ$ & $0.244255\, +  0.969711 i$ & $0.488511$ & $0.900301681$\\ \hline
$\circ$ &$ 0.295848\, +0.955235 i$& $0.591697$ & $0.501615796$\\ \hline
$\circ$ &$0.365525\, +  0.930802 i$ &$0.73105$ & $0.225388067$ \\ \hline
$\circ$ &$ 0.403589\, +0.91494 i$ & $0.807179 $ & $0.1361160029$\\ \hline
$\circ$ &$0.457052\, +  0.88944 i$ &$0.914104 $ & $0.0493695567$ \\ \hline
$\circ$ &$ 0.478014\, +0.878352 i$ & $0.956027 $ & $0.0235490782$\\ \hline
$\circ$ &$0.499085\, +  0.866553 i$ & $0.998171$ & $0.0009170979$\\ \hline
& $ -0.62071\pm 0.000834778 i$ & $-2.23176 \mp 0.00133189 i$ & $61.6152 \pm 18.99773 i$ \\ \hline
&$ -0.643412\pm 0.00246298 i$ & $-2.1976 \mp 0.00348646 i$  & $7.703480 \pm 0.6759767 i$ \\ \hline
&$ -0.693202\pm 0.00385058 i$ & $-2.13574 \mp 0.00416238 i$ &  $3.1296446 \pm 0.11879778 i$\\ \hline
&$ -0.779549\pm 0.00431503 i$ & $-2.0623 \mp 0.00278537 i$& $1.9275656 \pm 0.02619426 i$\\ \hline
 &$ -0.914251\pm 0.00229081 i$ & $-2.00804 \mp 0.000449854 i$ & $1.54165231 \pm 0.00241772 i$\\ \hline
  &$0.84037\, \pm 0.00372875 i$ & $2.0303 \mp 0.001551 i$& $-0.56939143 \pm 0.004072021 i$\\ \hline
   &$ 0.731022\, \pm 0.00427557 i$ & $2.09892 \mp 0.00372495 i$ & $-0.83854998 + 0.02192047 i$\\ \hline
   &$ 0.664454\, \pm 0.0032119 i$ & $2.16941 \mp 0.00406293 i$ & $-1.6862274 \pm 0.10076482 i$\\ \hline
   & $0.62906 \pm 0.00166081 i$ & $2.218722556 \mp 0.002536113 i$ & $-6.274506 \pm 0.9124470 i$\\ \hline
 &$0.61797$ & $2.23617$ & $1067.00$\\ \hline
\end{tabular}
\end{table}

In this case 
we see that the torsion polynomial is given by 
\[
\begin{split}
\sigma_{10}(t)
&=
t^{39}- 1400 t^{38}+400500  t^{37}-5.1428\times 10^7 t^{36}+3.5518\times 10^9  t^{35}\\
&-1.39149\times 10^{11}  t^{34}+3.1030\times 10^{12}  t^{33}-3.8975\times 10^{13} t^{32}+2.301\times 10^{14} t^{31}\\
&+4.36\times 10^{14} t^{30}-1.779\times 10^{16}  t^{29}+1.244\times 10^{17}  t^{28}-2.49\times 10^{17}  t^{27}\\
&-1.99\times 10^{18}  t^{26}+1.807\times 10^{19}  t^{25}-6.86\times 10^{19}t^{24}+1.25\times 10^{20}  t^{23}\\
&+2.6\times 10^{19} t^{22}-7.2\times 10^{20}  t^{21}+1.64\times 10^{21} t^{20}-1.09\times 10^{21} t^{19}\\
&-2.45\times 10^{21} t^{18}+6.5\times 10^{21}  t^{17}-5.1\times 10^{21}  t^{16}-2.8\times 10^{21}  t^{15}\\
&+8.8\times 10^{21} t^{14}-5.6\times 10^{21}  t^{13}-2.00\times 10^{21}  t^{12}+4.88\times 10^{21} t^{11}\\
&-2.09\times 10^{21} t^{10}-7.32\times 10^{20}  t^9+9.85\times 10^{20} t^8-2.549\times 10^{20} t^7\\
&-6.95\times 10^{19}  t^6+5.507\times 10^{19} t^5-1.2432\times 10^{19} t^4\\
&+1.21881\times 10^{18} t^3-4.8826\times 10^{16}  t^2+6.4321\times 10^{14} t-5.4976\times 10^{11}
\end{split}
\]

\section{problem}

In \cite{Kitano15-2,Kitano16-1}, 
the torsion polynomial for a Brieskorn homology 3-sphere 
obtained by surgeries along a torus knot, 
which is not exactly same with the one given in this paper, 
it can described by using Chebyshev polynomials of the first kind. 
It seems that it is natural, because any value of $\tau_\rho$ is given by some special values of the cosine function. 

In the case of the figure-eight knots, or in more general cases of hyperbolic knots, 
how to treat the torsion polynomial, it is a problem. 

\vskip 0.5cm
\textit{Acknowledgments}.\/
This study was starting while the author were visiting 
Aix-Marseille University from September, 2015 to March, 2016. 
The author would like to express their sincere thanks for Michel Boileau and Luisa Paoluzzi. 
He also thanks Joan Porti for pointing some errors in the first one. 
This research was supported in part by JSPS KAKENHI  25400101 and 16K05161.

\end{document}